\documentclass[11pt,reqno,twoside]{amsart}

\usepackage{amsmath}

\usepackage{amssymb}

\usepackage{epsfig}

\usepackage{color}

\numberwithin{equation}{section}

\newif\ifdraft\drafttrue

\draftfalse

\newcommand{\m}{{m}}
\newcommand{\n}{{n}}

\newcommand\mr{M_{m,n}}

\newcommand\ba{badly approximable}

\newcommand\ssm{\smallsetminus}

\newcommand\eq[2]{{\ifdraft{\ \tt [#1]}\else\ignorespaces\fi}\begin{equation}\label{eq:#1}{#2}\end{equation}}
\newcommand {\equ}[1]     {\eqref{eq:#1}}

\newcommand{\Q}{{\mathbb {Q}}}
\newcommand{\R}{{\mathbb {R}}}
\newcommand{\Z}{{\mathbb {Z}}}
\newcommand{\N}{{\mathbb {N}}}
\newcommand{\T}{{\mathbb{T}}}
\newcommand{\rd}{\mathbb{R}^d}
\newcommand{\rn}{\mathbb{R}^n}

\newcommand{\GL}{\operatorname{GL}}

\newcommand{\SL}{\operatorname{SL}}

\newcommand{\ggm}{G/\Gamma}

\newcommand{\diam}{\operatorname{diam}}
\newcommand{\dist}{\operatorname{dist}}
\newcommand{\diag}{\operatorname{diag}}
\newcommand{\supp}{\operatorname{supp}}
\newcommand{\const}{\operatorname{const}}

\newcommand {\ignore}[1]  {}

\newcommand\cic{C^\infty_{comp}}

\newcommand{\df}{{\, \stackrel{\mathrm{def}}{=}\, }}

\newcommand{\x}{{\mathbf{x}}}
\newcommand{\0}{{\mathbf{0}}}
\newcommand{\vv}{{\bf{v}}}

\newcommand{\vu}{{\bf u}}

\newcommand{\y}{{\mathbf y}}

\newcommand{\vm}{\mathbf{m}}
\newcommand{\eM}{\operatorname{M}}
\newcommand{\Bad}{\mathrm{Bad}}

\newcommand{\p}{{\bf p}}

\newcommand{\q}{{\mathbf{q}}}

\DeclareMathOperator{\codim}{codim}
\DeclareMathOperator{\spn}{span}
\DeclareMathOperator{\proj}{proj}
\DeclareMathOperator{\covol}{covol}

\newcommand{\vre}{\varepsilon}

\newcommand{\ca}{{\mathcal A}}

\newcommand{\cl}{{\mathcal L}}

\newcommand{\X}{\mathbf{X}}
\newcommand{\Y}{\mathbf{Y}}

\newtheorem{thm}{Theorem}[section]
\newtheorem{lem}[thm]{Lemma}
\newtheorem{prop}[thm]{Proposition}
\newtheorem{cor}[thm]{Corollary}

\newtheorem{remark}[thm]{Remark}
\newtheorem{defn}[thm]{Definition}

\begin{document}
\title[Dimension of uniformly B.A. systems of linear forms]{Dimension estimates for sets of 
uniformly badly approximable systems of linear forms}
\author{Ryan Broderick and Dmitry Kleinbock} 

\address{Department of Mathematics, Brandeis University, Waltham MA} 
\email{kleinboc@brandeis.edu}

\address{Department of Mathematics, Northwestern University, Evanston, IL} 
\email{ryan@math.northwestern.edu}

\begin{abstract}
The set of badly approximable $\m \times \n $ matrices is known
to have Hausdorff dimension $\m\n $. 
Each such matrix comes with its own approximation constant $c$,
and one can ask for the dimension of the set of badly approximable matrices
with approximation constant greater than or equal to some fixed $c$.
In the one-dimensional case, a very precise answer to this question is known. In this note,
we obtain upper and lower bounds in higher dimensions. 
The lower bounds are established via the technique of Schmidt games, while for the upper bound we use homogeneous dynamics methods, namely exponential mixing of flows on the space of lattices.
\end{abstract}

\maketitle

\section{Introduction}
For positive integers $\m $ and $\n $, 
let $\eM_{\m \times \n }$ denote the set of $\m \times \n $ matrices 
with real entries.
Each $A\in\eM_{\m \times \n }$ defines a linear transformation $\q\mapsto A\q $ from 
$\R^\n $ to $\R^\m $.
The components of this linear transformation can be regarded as a system of $\m $ linear forms in $\n $ variables. We will choose norms  $\|\cdot\|$  on $\R^n$ and $\R^m$, which will later without loss of generality be taken to be supremum norms.

\begin{defn}
\label{BAdef}
A matrix $A \in \eM_{\m \times \n }$ is said to be a \emph{badly approximable system of linear forms} if
there exists
$c > 0$ such that for all $\q \in \Z^\n  \smallsetminus \{0\}$ and $\p \in \Z^\m $,
\begin{equation}
\label{symmetricBA}
\|\q\|^\n\|A\q  - \p\|^\m \ge c.
\end{equation}
We write $\Bad_{\m ,\n }$ for the set of all badly approximable systems of linear forms.
\end{defn}
\noindent When $\n  = 1$, the elements of
this set are referred to as \textsl{badly approximable vectors} or, in the case $\m = \n  =1$,
\textsl{badly approximable numbers}.  

It was shown by W.\ Schmidt in \cite{S2} that 
$\dim(\Bad_{\m ,\n }) = \m \n $.\footnote{Here and throughout, $\dim$ stands for Hausdorff dimension.}
Note however that, since the constant $c$ in Definition \ref{BAdef} is allowed to depend on $A$,
$\Bad_{\m ,\n }$ is naturally written as a union over $c > 0$ of the sets
$$\Bad_{\m ,\n }(c) = \{A: \|\q\|^\n\|A\q  - \p\|^\m \ge c \text{ for all } \q \in \Z^\n  \smallsetminus \{0\}\text{ and }\p \in \Z^\m \}.$$
Schmidt's result can be restated as $$\sup_{c>0} \dim\big(\Bad_{\m ,\n }(c)\big) =  \lim_{c\to 0} \dim\big(\Bad_{\m ,\n }(c)\big) = \m\n  \,.$$
The asymptotics 
of the right hand side of the above expression as $c \to 0$ however
 is less well-studied.
It is more convenient in this context to discuss the codimension, which tends to $0$ as $c \to 0$,
so for a subset $S \subset \R^d$, write $\codim(S) = d - \dim (S)$. We wish to study
the rate at which $\codim\big(\Bad_{\m ,\n }(c)\big)$ tends to $0$ as $c \to 0$.
In the case $\m  =\n  = 1$,  J.\ Kurzweil proved in \cite{Ku} that the decay of 
$\codim\big(\Bad_{1,1}(c)\big)$
is linear. More precisely the following bounds are obtained for all sufficiently small $c > 0$:
$$.25c \le \codim\big(\Bad_{1,1}(c)\big) \le .99c.$$
Later, D.\ Hensley improved these estimates in \cite{H}. There he proves that
if $E_k$ is the set of real numbers whose continued fraction expansion involves
only partial quotients $\le k$, then 
$$\dim(E_k) = 1 - \frac{6}{\pi^2 k} - 72\frac{\log k}{\pi^4 k^2}+ O(1/k^2)\,.$$
But $\Bad_{1,1}\left(\frac{1}{k}\right) \subset E_k \subset \Bad_{1,1}\left(\frac{1}{k+2}\right)$
(see for example Theorem 1.9 in \cite{Bu}),
so we obtain more precise asymptotic dimension estimates for $\Bad_{1,1}(c)$.
However the authors are not aware of any nontrivial estimates in the literature
for $\codim\big(\Bad_{\m ,\n }(c)\big)$ when $\m $ or $\n $ is at least $2$.
In this article, we obtain upper and lower dimension bounds, though these estimates
do not align, so further research is needed to uncover the precise asymptotics in higher dimensions.
Specifically, we have the following:

\begin{thm}
\label{thmlower}
Given $\m , \n  \in \N$, there exist $p = p(\m ,\n ) > 0$ and $k_1 =k_1(\m ,\n ) > 0$ such that 
$$\codim\big(\Bad_{\m ,\n }(c)\big) \le k_1\frac{c^{1/p}}{\log(1/c)}\,.$$
In particular, one can take $p(\m ,1 ) = 2m$ and $p(1 ,\n ) = 2n^2$.
\end{thm}
\noindent The case $\max(m,n) > 1 $ is more involved; an explicit estimate for $p$ in that  case  is given in Theorem \ref{linearforms}.


\begin{thm}
\label{upper}
Given $\m , \n  \in \N$, there exists $k_2 = k_2(\m ,\n )  > 0$ such that
$$\codim\big(\Bad_{\m ,\n }(c)\big) \ge k_2 \frac{c}{\log(1/c)}\,.$$
\end{thm}

Combining the two theorems we have
\eq{2sided}{\m\n - k_1\frac{ c^{1/p}}{\log(1/c)} 
			\le  \dim\big(\Bad_{\m ,\n }(c)\big) \le \m\n - k_2 \frac{c}{\log(1/c)}\,.}

After this work was completed we became aware of a recent preprint \cite{W} of S.~Weil,
where he  obtained bounds on $\dim\big(\Bad_{\m,1}(c)\big)$ 
(the case of badly approximable vectors). 
Namely, it is proved there that for constants $k_1, k_2 > 0$ depending on $m$ one has
$$m - k_1\frac{c^{1/2m}}{\log(1/c)} \le \dim\big(\Bad_{\m,1}(c)\big) \le m - k_2\frac{c^{m+1}}{\log(1/c)}\,.$$
Comparing with \equ{2sided} and using the fact that one can take $p(\m,1) = 2m$ (Corollary \ref{codimbadd}), one sees that our lower bound agrees with Weil's and our upper bound is better.
Note that the methods of \cite{W} work in a more general set-up and can be applied also in other settings,
including badly approximable vectors with weights and bounded geodesics in 
hyperbolic manifolds.

Our lower bounds 
are obtained, like many lower dimension bounds of diophantine sets, using Schmidt's game. More precisely, we are employing a variant of Schmidt's game, the so-called hyperplane absolute game, see \S \ref{game}.
This is similar in spirit but still different from  the approach of \cite{W}. Consequently, our argument can be generalized to estimate the dimension of 
other sets whose union has the hyperplane absolute winning (HAW) property, such as the set of points whose trajectories under a total endomorphism miss a fixed open set.
See \cite{BFS} for a discussion of those sets
and a proof that they are HAW, as well as   \cite{AN} and \cite[Theorem 3.17]{W}, where dimension
bounds are obtained using other methods. It is also possible, following the ideas from  \cite{BFS} and similarly to \cite[Theorem 3.1]{W}, to produce lower  estimates for the dimension of the intersection of the aforementioned sets with fractals supporting    absolutely decaying  measures satisfying a power law (see \cite{KLW} for definitions).

Our upper  bounds on $\dim\big(\Bad_{\m ,\n }(c)\big)$ 
are based on homogeneous dynamics. Namely, we let 
\eq{ggm}{G = \SL_{\m +\n }( \R), \ \Gamma = \SL_{\m +\n }( \Z), \ X = \ggm\,,}
interpreting the latter space as the collection of unimodular lattices in $\R^{m+n}$.  This is a noncompact space of finite volume; we let $\mu$ denote the probability Haar measure on $X$. 
To $A\in \eM_{\m \times \n }$ we associate the   lattice $u_A\Z^{m+n}\subset X$, where
\eq{ua}{
u_A \df
\begin{pmatrix}
   I_m & A \\
    0 & I_n
       \end{pmatrix}\,.
}
Exploiting the  correspondence due to S.G.\ Dani \cite{dani, KM2},
we relate the set $\Bad_{\m ,\n }(c)$ to the set of orbits in $X$ 
which never enter a certain open subset. More precisely, we consider 
\eq{gt}{g_t \df \diag(e^{t/\m }, \ldots,
e^{ t/\m }, e^{-t/\n }, \ldots,
e^{-t/\n })\,,
}
and for $\vre>0$ define
\eq{ue}{
U_\vre \df \{\Lambda \in  X: \Lambda \cap B(\mathbf 0, \vre) \ne \{\0\}\}\,;
}
note that $X\ssm U_\vre$ is compact for every $\vre > 0$. 
Then it can be easily shown  (see Lemma~\ref{danicorr}) that for any $0 < c < 1$,  \eq{quantdani}{\Bad_{\m ,\n }(c) = \left\{A\in \mr :  \{g_tu_A\Z^{\m +\n }: t \ge 0\} \cap U_{\vre} = \varnothing\right\},} where $\vre = c^{\frac {1}{\m+\n}}$. 
We then 
use the exponential mixing property of the $g_t$-action
to produce many points whose orbits enter $U_\vre$, 
which makes it possible to estimate the number of small boxes needed to cover the set \equ{quantdani}.  This method is similar to the one used in \cite{KM1} to prove full dimension of the set of points with  bounded orbits of partially hyperbolic flows on arbitrary  homogeneous spaces $\ggm$.  Note however that the lower estimates  which can be extracted from that argument are weaker than what is produced by the method of Schmidt games in the case \equ{ggm}.

It is easy to see that the codimension of the set \equ{quantdani} in $\mr$ coincides with the codimension of the set
\eq{outsideofue}{
\left\{\Lambda\in X:  \{g_t\Lambda: t \ge 0\} \cap U_{\vre} = \varnothing\right\},} 
in $X$.  Note that $\mu(U_\vre)$ is asymptotically as $\vre \to 0$ equal to $\const\cdot \vre^{m+n}$, see \cite[Proposition 7.1]{KM2}. Thus, in the case  $m = n = 1$ the aforementioned result of Hensley shows that the codimension of the set \equ{outsideofue} is asymptotic to $\const\cdot\mu(U_\vre)$ as $\vre\to 0$. It is also worth mentioning a similar result of A.\ Ferguson and M.\ Pollicott \cite[Theorem 1.2]{FP}:  if $(J,f)$ is a {\it conformal repeller\/} (see \cite{FP} for definitions) and $z\in J$, then the Hausdorff codimension of the set of points of $J$ whose $f$-trajectories are disjoint from the ball $B(z,\vre)$  is, as $\vre \to 0$, asymptotic to a constant times the measure of the ball. 

In view of the aforementioned results it seems natural to conjecture that the value of $\codim\big(\Bad_{\m ,\n }(c)\big)$  is asymptotic to a constant times $c$; this means that the estimates in the left (resp.\ right)  hand side of \equ{2sided} can be significantly (resp.\ slightly) improved. Note however that the dynamical systems considered in \cite{H} and \cite{FP} admit a simple symbolic description, which is not the case for the partially hyperbolic flow studied in the present paper.

\medskip
{\em Acknowledgements.} We are grateful to M.\ Pollicott and M.\ Urbanski for useful discussions, to S.\ Weil for bringing our attention to his work \cite{W}, and to the referee of the previous version of this paper for helpful comments. The second-named author was supported
in part by NSF grant DMS-1101320.

\section{Lower estimates}\label{lower}

\subsection{Dimension estimates from the 
hyperplane absolute game}
\label{game}
To produce a lower estimate on $\dim\big(\Bad_{\m ,\n }(c)\big)$, we will use
the {\sl hyperplane absolute  game\/} introduced in \cite{BFKRW}. See that paper for an extensive
treatment of the game, which is a variant of Schmit's game introduced by C.\ McMullen in \cite{S1}.
We give only the definition below; this definition varies slightly from the one given 
in \cite{BFKRW}, but the class of HAW\footnote{The term `HAW' is an acronym
which stands for `Hyperplane Absolute Winning'.} sets (on $\rd$) that we obtain will be the same.
Complications arise when defining the game to be played on fractal subsets of $\rd$,
which we are able to avoid because we play only on $\rd$.

Given $d \in \N$, a set $S \subset \rd$ and a parameter $0 < \beta < 1/3$, 
the hyperplane absolute game is played by two players,
whom we will call Alice and Bob. A play of the game consists of these two players 
alternately choosing subsets of $\rd$.
For convenience, given a ball $B \subset \rd$, write $\rho(B)$ for its radius,
and given a set $S \subset \rd$ and $\varepsilon > 0$, write 
$S^{(\varepsilon)} = \{\x\in\rd : \dist(\x,S) \le \varepsilon\}$.
The game begins with Bob choosing a point $\x_0\in \rd$
and a radius $\rho_0 > 0$, thus specifying a closed ball $B_0 = B(\x_0, \rho_0)$.
Given an integer $i \ge 0$, if $B_i$ is chosen, Alice chooses a hyperplane $\cl_{i+1}$, 
and Bob must then choose a closed ball
$B_{i+1} \subset B_{i}$, which satisfies
\begin{enumerate}
\item $\rho(B_{i+1}) \ge \beta\rho(B_i)$, and
\item $B_{i+1} \cap \cl_{i+1}^{(\beta\rho(B_i))} = \varnothing$.
\end{enumerate}
We thus obtain a nested sequence of closed balls $B_0 \supset B_1 \supset \dots$.
If $$\cap_i B_i \cap S\neq \varnothing\,,$$ then Alice is declared the winner; otherwise Bob is.
We say that $S$ is a HAW set if for each $0 < \beta < 1/3$, Alice has a strategy to win the game regardless of Bob's choices. This HAW property has many consequences; in particular,
it implies that $\dim S = d$. (See \cite{BFKRW}.)

In \cite{BFS}, the set $\Bad_{\m ,\n }$ is shown to be HAW. However, for any $c > 0$,
$\Bad_{\m ,\n }(c)$ clearly does not have this property, since it is not dense so $B_0$ can be chosen disjoint from it. The union $\Bad_{\m ,\n } = \cup_{c>0} \Bad_{\m ,\n }(c)$ is proven to be HAW
by choosing a $c > 0$ dependent on $\x_0$, $\rho$, and $\beta$, and tailoring Alice's
strategy in a game with these parameters to ensure that $\cap B_i  \cap  \Bad_{\m ,\n }(c)\neq \varnothing$.
There, the $c$ that was chosen and its relationship to the parameters of the game
were irrelevant, but for our proofs they happen to be crucial, so we introduce the following definitions.

\begin{defn}
We say $S$ is \emph{$(\x,\rho,\beta)$-HAW} if Alice has a strategy to win the hyperplane absolute game
with parameter $\beta$ provided Bob's initial move is centered at $x$ and has radius equal to $\rho$.

If $S$ is $(\x,\rho,\beta)$-HAW for each $\x \in \rd$ and each $\rho > 0$, we say it is
$\beta$-HAW.
\end{defn}

Thus, $S$ is HAW if and only if
$S$ is $\beta$-HAW for each $0 < \beta < 1/3$, or equivalently,
if $S$ is $(\x,\rho,\beta)$-HAW for each $0 < \beta < 1/3$, each $\rho > 0$, and each $\x \in \R^d$.

Now, the full dimension of HAW sets follows from the fact that they are all
$\alpha$-winning sets for Schmidt's game. See \cite{S1} for a definition
of the $\alpha$-winning and $(\alpha,\beta)$-winning properties. Analogously to the above,
a set is said to be $\alpha$-winning if it is $(\alpha,\beta)$-winning for all $0 < \beta < 1$,
and indeed one can show that for $0 < \beta < 1/3$, $\beta$-HAW sets are all $(1/3,3\beta)$-winning.
Schmidt proves that $\alpha$-winning subsets of $\rd$ have full dimension
by first estimating the dimension of $(\alpha,\beta)$-winning subsets of $\rd$,
so we get a dimension bound on $\beta$-HAW sets directly from \cite{S1}.
However, this bound does not suffice for our purposes, so we prove the following stronger
estimate which holds for the smaller class of $\beta$-HAW sets. We will use this theorem 
to obtain Theorem \ref{thmlower}.

\begin{thm}
\label{dimbound1}
There exists a constant $M_d$ depending only on $d \in \N$ such that
if $S$ is an $(\x,\rho,\beta)$-HAW subset of $\R^d$ and $B = B(\x,\rho)$ 
then $$\dim(S \cap B) 
\ge d - \frac{\log(1-M_d\beta)}{\log \beta}.$$
\end{thm}

\begin{proof}
Suppose Bob chooses $B_0 = B(\x,\rho)$.
We will use Alice's winning strategy to
construct a Cantor-like subset of $B_0 \cap S$ with the required dimension.
Let $C_0$ be the hypercube inscribed in $B_0$.
We will define $C_k$ in such a way that the following hold:
\begin{enumerate}
\item $C_k = \bigcup_{i=1}^{n_k} C_{k,i}$, where the $C_{k,i}$ are
	essentially disjoint hypercubes of side length $\frac{2\beta^k\rho}{\sqrt{d}}$.
\item For each $k$ and $i$, there is a unique $j$ such that
	$C_{k,i} \subset C_{k-1,j}$.
\item For some constant $M_d$, each
	hypercube $C_{k,i}$ contains at least $\beta^{-d} - M_d\beta^{-d+1}$
	hypercubes of the form $C_{k+1, j}$.
\item Given $C_{k,i_k} \subset C_{k-1,i_{k-1}} \subset \dots \subset C_{0, i_0}$,
	the balls $B_{j}$ superscribing $C_{j,i_{j}}$ define the first $k+1$ moves
	of a legal play of the HAW game in which Alice uses her winning strategy.
\end{enumerate}
It follows from these properties that $C \df \cap C_k \subset S$.
Suppose $C_0, \dots, C_k$ are constructed so that the above hold and let $1 \le i \le n_k$.
Let $B_0, \dots, B_k$ be the initial play of the game corresponding to $C_{k,i}$,
which is well-defined by properties (2) and (4) above.
Take $\cl$ to be the hyperplane dictated by Alice's 
winning strategy. Note that $C_{k,i}$ contains
$(\beta^{-1}-1)^d$ essentially disjoint hypercubes of side length 
$\frac{2\beta^{k+1}\rho}{\sqrt{d}}$,
and the ball superscribing each such hypercube has radius $\beta^{k+1}\rho$ 
and is contained in $B_k$.
If, additionally, a given hypercube has distance at least $2\beta^{k+1}\rho$
from $\cl$, then since the distance from the center of the hypercube to its boundary is
$\frac{\beta^{k+1}\rho}{\sqrt{d}}$, the ball superscribing it will 
have distance greater than
$$2\beta^{k+1}\rho - \left(\beta^{k+1}\rho - \frac{\beta^{k+1}\rho}{\sqrt{d}}\right) 
	\ge \beta^{k+1}\rho \ge \beta\rho(B_k)$$
from $\cl$. Thus, this ball will be disjoint
 from $\cl^{(\beta\rho(B_k))}$ and will therefore
be a legal move for Bob in the hyperplane absolute game.
There are at least
$$(\beta^{-1}-1)^d - (2\sqrt{d}+1)(\sqrt{d}\beta^{-1}+1)^{d-1}$$
hypercubes $C_{k+1,j} \subset C_{k,i}$ with the required distance from $\cl$.
Including all such hypercubes in $C_{k+1}$ for each $i$ guarantees that properties 
(1)-(4) above are satisfied for 
$C_0, \dots, C_{k+1}$, if we take
$M_d =  d\cdot d! +(2\sqrt{d})^{d-1}(2\sqrt{d}+1)$ so that 
\begin{eqnarray*}
\beta^{-d} - M_d\beta^{-d+1} &=& \beta^{-d} - d\cdot d! \beta^{-d+1}
	-(2\sqrt{d}+1)(2\sqrt{d}\beta^{-1})^{d-1}\\
	&\le&  (\beta^{-1}-1)^d - (2\sqrt{d}+1)(\sqrt{d}\beta^{-1}+1)^{d-1} .
\end{eqnarray*}
(Here, we use the fact that $\beta < 1/2$.)
Thus, the induction continues and we obtain a fractal subset $C \subset S \cap B_0$.

Since, at each stage in the construction of $C$ we keep at least $\beta^{-d} - M_d\beta^{-d+1}$
stage-$(k+1)$ hypercubes within each stage-$k$ hypercube, and
the diameters are scaled down by $\beta$,
\begin{eqnarray*}
\dim(S\cap B_0) \ge \dim(C) &\ge& \frac{\log(\beta^{-d} - M_d\beta^{-d+1})}{\log \beta^{-1}}\\
  &=& \frac{-d\log(\beta) + \log(1-M_d\beta)}{-\log \beta} = d - \frac{\log(1-M_d\beta)}{\log \beta}.
\end{eqnarray*}
\end{proof}

As an immediate consequence of the previous theorem,
we can uniformly bound from below the Hausdorff dimension of any $\beta$-HAW
set within any open set $U$:

\begin{cor}
\label{dimbound2}
If $M_d$ is as in Theorem \ref{dimbound1}, $S$ is a $\beta$-HAW subset of $\R^d$,
and $U \subset \R^d$ is open, 
then $$\dim(S \cap U) \ge d - \frac{\log(1-M_d\beta)}{\log \beta}.$$
\end{cor}

As another direct corollary, we have the following bound on the decay rate of the codimension of 
$(\x,\rho,\beta)$-HAW sets.

\begin{cor}
\label{codimbound}
Let 
$$c_d(\beta) 
	= \sup\{\codim(S) : S \subset \rd \text{ is $(\x,\rho,\beta)$-HAW 	
		for some }\x \in\rd, \rho > 0\}.$$
Then $c_d(\beta) = O\left(\frac{\beta}{\log(1/\beta)}\right)$ as $\beta \to 0$.
\end{cor}

\begin{proof}
From Theorem \ref{dimbound1}, it follows that
$$c_d(\beta) \le \frac{\log(1-M_d\beta)}{\log \beta}.$$
Hence,
$$c_d(\beta)\cdot\frac{\log(1/\beta)}{\beta} 
		\le \frac{\log(1-M_d\beta)}{-\beta} \longrightarrow M_d < \infty.$$
\end{proof}

\subsection{A lower dimension bound for $\Bad_{\m ,\n }(c)$}
We now apply Theorem \ref{dimbound1} to deduce Theorem \ref{thmlower}.
To do so we will need to obtain the $(\x,\rho,\beta)$-HAW property for
$\Bad_{\m ,\n }(c)$, and carefully note the dependence of the parameters $\x$, $\rho$,
and $\beta$ on the approximation constant $c$.
The case $\n =1$ is easier.
\begin{thm}
\label{badd}
$\Bad_{\m, 1}(c) \subset \R^\m$ is $(\x,1,\beta)$-HAW
	whenever $c^{1/\m} \le \frac{\beta^2}{4m!}$, $\x\in\R^\m$, 
	and $0 < \beta < 1/3$.
\end{thm}

In \cite{BFKRW}, it was proved that $\Bad_{\m, 1}(c)$ is HAW. Our proof of Theorem
\ref{badd} uses the same basic strategy but yields a better bound on $\beta$ in terms of $c$.\footnote{We thank the referee for suggesting this improvement.}
We will use the `simplex lemma,' which was proved by Davenport and
appears in \cite{KTV} in a form which, in particular, implies the following.

\begin{lem}[Simplex Lemma]
\label{simplex}
Let $m \in \N$, $r > 0$, and $\x \in \R^m$. Then the rational points in $B(\x, r)$
with denominator at most $(2m!)^{-\frac{m}{m+1}}r^{-\frac{m}{m+1}}$ all lie in a single hyperplane.
\end{lem}

\begin{proof}[Proof of Theorem \ref{badd}]
Let $0 < \beta < 1/3$ and fix $\rho_0 = 1$.
Let $B(\x_k, \rho_k)$ denote the ball Bob chooses on his $(k+1)$st turn.
By Lemma \ref{simplex}, the rationals in $B(\x_k, 2\rho_k)$ with denominator at most
$(2m!)^{-\frac{m}{m+1}}(2\rho_k)^{-\frac{m}{m+1}}$ all lie in a single hyperplane $\cl_{k+1}$.
Alice will choose this hyperplane as her $(k+1)$st move in the game.
Then every rational with denominator at most $(2m!)^{-\frac{m}{m+1}}(2\rho_k)^{-\frac{m}{m+1}}$
is either outside $B(\x_k, 2\rho_k)$ in which case its distance from $B(\x_{k+1}, \rho_{k+1})$
is at least $\rho_k \ge \beta\rho_k$, or it is in $\cl_{k+1}$ in which
case its distance from $B(\x_{k+1}, \rho_{k+1})$ is at least $\beta \rho_k$.
Let $\x \in \cap B(\x_k, \rho_k)$,  $\p \in \Z^m$, and $q \in \N$.
Since $$q \ge 1 > (2m!)^{-\frac{m}{m+1}} > (2m!)^{-\frac{m}{m+1}}(2\rho_0)^{-\frac{m}{m+1}},$$
there is a unique $k \in \N$ such that
$$(2m!)^{-\frac{m}{m+1}}(2\rho_k)^{-\frac{m}{m+1}} \ge q 
	> (2m!)^{-\frac{m}{m+1}}(2\rho_{k-1})^{-\frac{m}{m+1}}
		\ge (2m!)^{-\frac{m}{m+1}}\left(\frac{2\rho_k}{\beta}\right)^{-\frac{m}{m+1}}.$$
By the above, Alice's strategy guarantees that
$$d(\x, \p/q) \ge \beta \rho_k \ge \beta (\beta/2) (2m!)^{-1} q^{-\frac{m+1}{m}} 
		= \frac{\beta^2}{4m!} q^{-\frac{m+1}{m}}.$$
This can be rewritten as $$\|q\x - \p\|^m \ge \left(\frac{\beta^2}{4m!}\right)^m q^{-1}\ge cq^{-1}\,,$$ finishing the proof.
\end{proof}

\noindent Combining Theorem \ref{badd} and Corollary \ref{codimbound}, 
we get the following.

\begin{cor}
\label{codimbadd}
For any $\m \in \N$, $\codim\big(\Bad_{\m, 1}(c)\big) 
	= O\left(\frac{c^{\frac{1}{2\m}}}{\log(1/c)}\right)$ as $c \to 0$.
\end{cor}

Before proceeding to the general case, we use Khintchine's transference principle to obtain an estimate for the $\m=1$ case, as this method provides a tighter
bound than the one we get from taking $m=1$ in the general theorem below.
Specifically, we will use the following (for a proof, see \cite[Chapter V, \S 2]{C}):

\begin{prop}
\label{transference}
For each $\m, \n \in \N$ and each $c > 0$, $\Bad_{\m,\n}(c) \supset \Bad_{\n,\m}(c')$,
where $c' = \const(\m, \n) c^{\frac{1}{\m+\n-1}}$.
\end{prop}

\noindent Applying Proposition \ref{transference} and Theorem \ref{badd}, we get

\begin{cor}
For any $\n \in \N$, $\codim\big(\Bad_{1, \n}(c)\big) = O\left(\frac{c^{\frac{1}{2\n^2}}}{\log(1/c)}\right)$.
\end{cor}
One can also use the hyperplane absolute game to obtain a dimension bound
	for general $\m, \n \in \N$.

\begin{thm}
\label{linearforms}
For $\m, \n \in \N$,  Theorem \ref{thmlower} holds with
\eq{defp}{p = p(\m ,\n ) \df \left(\m (\m +\n ) + \n (\m +\n )^3\right)\cdot\max\left\{\frac{4\n +1}{\m }, \frac{4\m +1}{\n}\right\}.}
In other words, with $p$ as in \equ{defp} one has
$\codim\big(\Bad_{\m , \n }(c)\big) = O\left(\frac{c^{1/p}}{\log(1/c)}\right)$ as $c \to \infty$.
\end{thm}

\begin{remark}
Note that the 
expression  \equ{defp} is not symmetric in $\m$ and $\n$, 
so that when $\n > \m$, we get a weaker bound for $\Bad_{\m, \n}(c)$
than we do for $\Bad_{\n, \m}(c)$. However, Proposition \ref{transference}
does not help in this case, as $(\m+\n-1)p(\n,\m) \ge p(\m,\n)$ for any
$\m, \n \in \N$.
\end{remark}

\begin{proof}[Sketch of proof of Theorem \ref{linearforms}] The method of proof  constitutes a quantitative refinement of the argument of  \cite{S2} and \cite{BFS}; thus we will only give a sketch. 
In \cite{S2}, Schmidt proved that the set $\Bad_{\n, \m}$ is 
winning, and 
in \cite{BFS}, by applying the same scheme, this set was shown to be HAW.
In both cases, the 
basic strategy
(see \cite[\S 4]{S2} and \cite[Lemma 5.3]{BFS})
is to prove that for any $\beta > 0$ there exists $R > 0$ such that the set $\ca_R$ of matrices in $\mr$, for which a certain system\footnote{see \cite[(5.23)--(5.26)]{BFS}} of inequalities involving $R$ has no nontrivial integer solutions, is 
winning for the game played with parameter $\beta$. One can show that the proof  in \cite{BFS}  implies that the set $\ca_R$ is $(\mathbf{0},1,\beta)$-HAW, where   $R$ 
can be taken to be equal to
$K \beta^{-\ell(m,n)}$,
with $K$ depending only on $m$, $n$ and the initial ball of the game (but not $\beta$),
and with $$\ell(m,n) = \max\left\{\frac{4\n +1}{\m }, \frac{4\m +1}{\n}\right\}.$$
Furthermore, $\ca_R$ can be shown to be a subset of $\Bad_{m,n}(c)$, where 
$$c = K^{\prime}(m,n) \cdot \beta^{-\ell(n(m+n) - m(m+n)^3)}$$ (here $K^{\prime}$ depends only on $m$ and $n$). By Corollary \ref{codimbound}, 
Theorem \ref{thmlower} follows.
\ignore{
(works for the set of $A$ that satisfy 
(\ref{system3})--(\ref{system4}) in addition to (\ref{system1})--(\ref{system2}), 
as this is required for the induction argument in the proof
to continue. To obtain an estimate on the parameter $\beta$ we may choose
for a given approximation constant $c$, we state a quantitative version of a lemma in 
\cite{BFS}. 
This theorem follows from the arguments in \cite{BFS}:

The proof of the main theorem in that paper then implies that 
$\Bad_{m,n}(c)$ is $(\mathbf{0},1,\beta)$-HAW, where 
$c = K^{\prime}(m,n) \cdot \beta^{-\ell(n(m+n) - m(m+n)^3)}$. By Colollary \ref{codimbound}, 
Theorem \ref{thmlower} follows.}
\end{proof}

\section{Upper estimates}\label{dmitry}

In this section we prove Theorem \ref{upper}.
We begin by recasting the definition of $\Bad_{\m ,\n }(c)$ using homogeneous dynamics. 
In this section $\|\cdot\|$ will always denote the supremum norm
and all distances in $\R^k$ will be induced from this norm, so balls will in fact be cubes. 
Let 
$G$, $\Gamma$ and $X$ be as in \equ{ggm}. Fix a right-invariant Riemannian metric on $G$ and let `$\dist$' denote the associated distance
function, both on $G$ and on $X$.
 Also let  $u_A$ and $g_t$ be as in \equ{ua} and \equ{gt},
\ignore{ $$g_t \df \diag(e^{t/\m }, \ldots,
e^{ t/\m }, e^{-t/\n }, \ldots,
e^{-t/\n })\,,
$$ and for $A\in \eM_{\m \times \n }$ let 
${
u_A \df
\begin{pmatrix}
   I_m & A \\
    0 & I_n
       \end{pmatrix}
}
$. } and let us   denote $$H = \{
u_A : A\in \eM_{\m \times \n }\}\,;$$ this is the expanding horospherical subgroup relative to $g_1$. Dani \cite{dani} proved that $A\in \eM_{\m \times \n }$ is \ba\ if and only if the trajectory $\{g_tu_A\Z^{\m +\n }: t > 0\}$ is bounded. It is not hard to make this equivalence quantitative. 
With the `cusp neighborhood' $U_\vre$ defined as in \equ{ue},  one has

\begin{lem}
\label{danicorr}
For any $0 < c < 1$, $A\in \Bad_{\m ,\n }(c)$  if and only if $$\{g_tu_A\Z^{\m +\n }: t > 0\} 
\cap U_\vre = \varnothing
$$ where $\vre = c^{\frac {1}{\m+\n}}$.  
\end{lem}

\begin{proof}
	First note that $g_t u_A \Z^{\m +\n }$ consists of vectors of the form
	\begin{equation*}
	\begin{split}
	\begin{pmatrix}
   e^{t/\m }I_\m  & 0 \\
    0 & e^{-t/\n }I_\n 
       \end{pmatrix}
       \begin{pmatrix}
   I_\m  & A \\
    0 & I_\n 
       \end{pmatrix}
       \begin{pmatrix}
   \p \\
    -\q
       \end{pmatrix}
       &=
              \begin{pmatrix}
   e^{t/\m }I_\m  & e^{t/\m }A \\
    0 & e^{-t/\n }I_\n 
       \end{pmatrix}
       \begin{pmatrix}
   \p \\
    -\q
       \end{pmatrix} \\
      & =
       \begin{pmatrix}
       e^{t/\m }\p - e^{t/\m }A\q\\
       -e^{-t/\n }\q       
       \end{pmatrix}
       \end{split}
       \end{equation*}
      Suppose 
      $\|A\q - \p\| \ge c^{1/\m}\|\q\|^{-\n /\m }$ for all 
      $\p \in \Z^\m $ and $\q\in \Z^\n  \smallsetminus\{\0\}$, 
      i.e.\ $A \in \Bad_{m,n}(c)$.
      We claim that $\left\| \begin{pmatrix}
       e^{t/\m }\p - e^{t/\m }A\q\\
       -e^{-t/\n }\q       
       \end{pmatrix}\right\| \ge \vre$ for all $t > 0$ and all nonzero $(\p, \q) \in \Z^m \times \Z^n$.
       Indeed, if $\q = \0$ and $\p \neq \0$,
       the norm of this vector is equal to $$e^{t/m}\|\p\| \ge e^{t/m} \ge 1 \ge \vre\,.$$
       Now take $\q \neq \0$ and suppose 
       $\| -e^{-t/\n }\q  \|  < \vre$, so that $\|\q\|^{-\n /\m } \ge e^{-t/\m }c^{\frac{-\n }{\m(\m +\n) }}$.
       Then 
       $$\|e^{t/\m }\p - e^{t/\m }A\q\| = e^{t/\m }\|\p - A\q\| \ge e^{t/\m } c^{1/\m} \|\q\|^{-\n /\m } 
       		\ge c^{\frac{1}{\m}}\cdot c^{\frac{-\n }{\m(\m +\n)}} = \vre.$$
	
	Conversely, if $\left\| \begin{pmatrix}
       e^{t/\m }\p - e^{t/\m }A\q\\
       -e^{-t/\n }\q       
       \end{pmatrix}\right\| \ge \vre$ for all $\p \in \Z^{\m }$ and $\q\in \Z^{\n }\smallsetminus\{\0\}$,
       fix such $\p$ and $\q$ 
       and let $t$ be such that $e^{-t/\n }\|\q\| = \vre = c^{\frac{1}{\m +\n }}$,
       so $\|\q\|^\n = e^t c^{\frac{\n}{\m+\n}}$.
        Then we must have $e^{t/\m}\|\p -A\q\| \ge \vre$
        and hence 
        $$\|\q\|^n \|A\q-\p\|^m \ge (e^tc^{\frac{\n}{\m+\n}})(e^{-t}\vre^m)
				= c^{\frac{\n}{\m+\n} + \frac{\m}{\m+\n}} = c.$$     
       Since $\p$ and $\q$ were arbitrary,
       $A \in \Bad_{\m ,\n }(c)$.
       \end{proof} 

Our strategy for proving the theorem will be to construct a covering of $\Bad_{\m ,\n }(c)$ 
by small boxes, thereby bounding from above the box dimension of the set.
 We are going to fix (small, depending on $c$) positive $r$  and restrict our attention to a ball $B = B(\mathbf 0, r/2)$  in  $\eM_{\m \times \n }$. Also fix (large) $t > 0$, (small) $\delta> 0$ and a lattice $\Lambda\in X$, and consider
$$
\mathcal{A}(B,t,\delta,\Lambda) \df \{A\in B : g_tu_A\Lambda \in U_\delta\}\,.
$$
To estimate the measure of this set (from both sides), one can use exponential decay of matrix coefficients. Specifically, we need the following, which appears, in slightly different form,
as Proposition 2.4.8 in \cite{KM1}:

\begin{prop}
\label{expdecay} 
There exist $E, \lambda, k, \ell > 0$ such that  for any $f\in \cic(H)$, for any $\psi\in
\cic(X)$ such that the map $g\mapsto g\Lambda$ is injective on some ball in $G$ containing $\supp\,f$,
 and for any $t\ge 0$ one has 
$$
\left|\int_{H}
f(A)\psi(g_thx)\,d\nu(h) - \int_{H}
f\,d\nu \int_X\psi\,d\mu\right|\le \const(f,\psi) e^{-\lambda
t}\,,
$$
where $$\const(f,\psi) = E \|f\|_\ell \|\psi\|_\ell \left( \max_{x\in X} \|\nabla \psi(x)\| \cdot \int_{H}
|f|\,d\nu\right)^k \,.$$
\end{prop}

Here $\nu$ is Haar  measure on $H$, corresponding to Lebesgue measure on $\eM_{m\times n}$,  $\mu$ is probability Haar measure on $X$, and $\|\cdot\|_\ell$ are Sobolev norms.
Note that the statement in \cite{KM1} is somewhat different, as the constant
is not stated explicitly, but the proof produces exactly this constant in the case
that $g \mapsto g\Lambda$ is injective on some ball in $G$ containing $\supp\,f$, which
we assume here.
\ignore{
This is not quite how the proposition is stated there, but look at the last line of its proof. So perhaps, if we decide to do it, we'll need to \textcolor{red}{repeat the proof} using   \cite{KM1}[Theorem 2.4.5] as the source of information. 
}

We use Proposition \ref{expdecay} to deduce the following measure estimate.

\ignore{
\begin{prop}
\label{expdecaysets} 
There exist constants $0 < D < 1$ and $E', a > 0$ such that  for  $B$   as above, for small enough $\delta > 0$, for $\Lambda$ such that  
\eq{inj}{\text{the map $h\mapsto h\Lambda$ is injective on some ball in $G$ containing\footnote{Here we identify $B$ with its image under the map $Y\mapsto u_Y$.} }B\,,} and for any $t\ge 0$ one has 
$$
\left|\nu\big(\mathcal{A}(B,t,\delta,\Lambda)\big) - Dr^{mn}\delta^{m+n}\right|\le E' r^{-a} e^{-\lambda
t}\,.
$$
\end{prop}
}

\ignore{
\begin{prop}
\label{expdecaysets} 
There exist constants $D, E', a, \lambda' > 0$ such that  for  $B$   as above, for small enough $\delta > 0$, for $\Lambda$ such that  
\eq{inj}{\text{the map $h\mapsto h\Lambda$ is injective on some ball in $G$ containing\footnote{Here we identify $B$ with its image under the map $Y\mapsto u_Y$.} }B\,,} and for any $t\ge 0$ one has 
$$
\left|\nu\big(\mathcal{A}(B,t,\delta,\Lambda)\big) - Dr^{mn}\delta^{m+n}\right|\le E' r^{-a}e^{-\lambda't}\,.
$$
\end{prop}
}

\begin{prop}
\label{expdecaysets} 
There exist constants $0 < D < 1$, $\lambda' > 0$ and $E' > 1$ such that  for  $B$   as above, for $0 < r < 1/2$, for small enough $\delta > 0$, for $\Lambda$ such that  
\eq{inj}{\text{the map $h\mapsto h\Lambda$ is injective on some ball in $G$ 
containing\footnotemark\ } B\,,} and for any $t\ge 0$ one has 
\eq{conclusion}{
\footnotetext{Here we identify $B$ with its image under the map $Y\mapsto u_Y$.} 
\nu\big(\mathcal{A}(B,t,\delta,\Lambda)\big) 
\ge  Dr^{\m \n }\delta^{\m +\n } -  E' e^{-\lambda' t}\,. 
}
\end{prop}

To prove this, we will apply Proposition \ref{expdecay} to smooth approximations 
of $1_B$ and $1_{ U_\delta}$. In order to extract useful information
from the conclusion of the proposition though we will need to bound the Sobolev norms
of these approximations, so we first prove the following lemmas.

\ignore{
\begin{lem}
\label{Sobapprox1}
For each $r > 0$, there exist smooth functions $F_\vre : \R \to \R$ ($\vre > 0$) such that
	\begin{enumerate}
		\item $F_\vre \ge 0$
		\item $F_\vre |_{(-\infty,r]} = 1$
		\item $F_\vre |_{[r+\vre,\infty)} = 0$
		\item $\|F_\vre\|_\ell \in O(\vre^{-(\ell+1)})$ for each $\ell \in \N$.
	\end{enumerate}
\end{lem}

\begin{proof}
Let 
$$g(x) = \begin{cases}
e^{\frac{-1}{(1-x)^2}} & |x| < 1\\
0                                & |x| \ge 1,
\end{cases}$$
let $h = 1_{(-\infty,r+\frac12\vre]}$, and let $g_\vre = \frac{1}{2\vre}g\left(\frac{x}{2\vre}\right)$.
Then $g_\vre$ is smooth, 
so the convolution $F_\vre = g_\vre \ast h$ is smooth and indeed $F_\vre^{(n)} = g_\vre^{(n)}\ast h$.
It is clear that $F_\vre$ satisfies properties properties (1)--(3).
Let $c = \displaystyle\max_{0 \le i \le \ell} \max_{x\in \R} |g^{(i)}(x)|$.
Note that $g_\vre^{(n)}(x) = (2\vre)^{-n-1} g^{(n)}(x)$, so for each $\vre > 0$
$\displaystyle\max_{0 \le i \le \ell} \max_{x\in \R} |g_\vre^{(i)}(x)| \le c\vre^{-(\ell+1)}$.
Furthermore, $g_\vre^{(i)}(x) = 0$ for all $x \not\in [r,r+\vre]$ and all $i \ge 1$.
Property (4) follows.
\end{proof}
}

\begin{lem}
\label{Sobapprox2}
For any $\ell, k \in \N$ there exists $M \in \R$ such that 
for any $0 < r \le 1/2$, 
there exist $C^\infty$ functions $f_\vre : \R^k \to \R$ ($0 < \vre \le r$) such that 
	\begin{enumerate}
		\item $0 \le f_\vre \le 1$
		\item $f_\vre |_{B({\mathbf 0},r)} = 1$
		\item $f_\vre |_{B({\mathbf 0},r+\vre)^c} = 0$
		\item $\|f_\vre\|_\ell \le M \vre^{-(k + \ell +1)}$.
	\end{enumerate}
\end{lem}

\begin{proof}
Let $g : \R^k \to [0,\infty)$ be a smooth function with $\supp g \subset B(\0,1)$ and $\|g\|_{L^1} = 1$,
and let $g_\vre(\x) = (2/\vre)^k g(2\x/\vre)$.
Then the convolution
$f_\vre = g_\vre \ast 1_{B(\0, r + \vre/2)}$
is smooth and indeed, for any multi-index $\alpha = (\alpha_1, \dots, \alpha_k)$,
 $$\dfrac{\partial^{|\alpha|} f_\vre}{\partial x_1^{\alpha_1}\dots \partial x_k^{\alpha_k} }
 = \dfrac{\partial^{|\alpha|} g_\vre}{\partial x_1^{\alpha_1}\dots \partial x_k^{\alpha_k}} \ast h.$$
 We will write $D_\alpha$ for 
 $\dfrac{\partial^{|\alpha|}}{\partial x_1^{\alpha_1}\dots \partial x_k^{\alpha_k}}$.
It is easy to see that properties (1)--(3) are satisfied.
Now set
$$c =  \displaystyle\max_{0 \le |\alpha| \le \ell} \max_{\x\in \R^k} \left|D_\alpha g(\x)\right|.$$
Then 
$$| D_\alpha g_\vre(\x)|
	= |(\vre/2)^{-k-|\alpha|-1}D_\alpha g(2\x/\vre)| \le 2^{k+\ell +1}c\vre^{-(k+\ell+1)}.$$
Thus, since $D_\alpha g_\vre$ is supported on $B(\0, \vre/2) \subset B(\0, 1/2)$,
it follows that $$\|D_\alpha g_\vre\|_{L^2} \le 2^{k+\ell +1}c\vre^{-(k+\ell+1)}\,.$$
Hence, applying Young's Inequality, we have that
$$\|D_\alpha f_\vre\|_{L^2} = \| D_\alpha g_\vre \ast 1_{B(\0,r+\vre/2)}\|_{L^2} 
		\le  \|D_\alpha g_\vre\|_{L^2} \|1_{B(\0,r+\vre/2)}\|_{L^1}
		\le 2^{k+\ell +1}c\vre^{-(k+\ell+1)},$$
from which property (4) follows.
\end{proof}

To approximate $1_{ U_\delta}$, we will use the following lemma.

\begin{lem}
\label{Sobapprox3}
There exists constants $C, D_1, D_2 > 0$ depending only on $m,n$ such that for every small enough 
$\delta > 0$ one can find
$\psi : X \to [0,\infty)$ with the following properties:
\begin{enumerate}
	\item $\psi \le 1_{ U_\delta}$; 
	\item $\psi \in \cic(X)$;
	\item $D_1\delta^{\m +\n } \le \int \psi\, d\mu \le \int 1_{ U_\delta}\, d\mu \le D_2\delta^{\m +\n }$;
	\item $\|\nabla \psi\| \le 
	C$.
\end{enumerate}
Moreover,   for each $\ell \in \N$ there exists $C_\ell > 0$ such that  for any $\psi$ as above one has   $\|\psi\|_\ell \le C_\ell\int \psi\, d\mu$.
\end{lem}

We remark that this is essentially a consequence of
Lemma 4.2 and Proposition 7.1 in \cite{KM2}, together with the inner regularity
of Haar measure on open sets, but we provide a proof for completeness.

\begin{proof} It follows directly from Lemma 4.2 in  \cite{KM2} that there exist $c_1, c_2 > 0$ such that for any $\delta$ one has
$$c_1 \delta^{n+m} \ge \mu( U_\delta) \ge c_1\delta^{n+m} - c_2\delta^{2(n+m)}.$$
Thus for small enough $\delta$ we get
$$c_1 \delta^{n+m} \ge \mu( U_\delta) \ge \frac{c_1}{2}\delta^{n+m}.$$
Choose $\vre > 0$ such that whenever $\dist(\Lambda,\Lambda') < \vre$,
the smallest nonzero vector in $\Lambda$ is least half as large as the smallest nonzero vector in $\Lambda'$ (clearly $\vre$ depends only on the choice of the metric). Also denote
$$A_\delta'(\vre) \df \{\Lambda \in  U_\delta : \dist(\Lambda, \partial  U_\delta) > \vre\}.$$
Then it follows that $U_{\delta/2} \subset A_\delta'(\vre) \subset  U_\delta$
and therefore  for small enough $\delta$ we have
$\mu\big(A_\delta'(\vre)\big) \ge \frac{c_1}{2}(\delta/2)^{n+m}$.
Note that $A_\delta'(\vre)$ is an open set, so 
since Haar measure is inner regular on open sets,
there exists a compact $A_\delta(\vre) \subset A_\delta'(\vre)$
with $$\mu\big(A_\delta(\vre)\big) \ge \frac12 \mu\big(A_\delta'(\vre)\big) \ge 2^{-(n+m+2)}c_1\delta^{n+m}\,.$$
Let $A_\delta(\vre)^+$ and $A_\delta(\vre)^{++}$ be the closed $\vre/4$ and $3\vre/4$ neighborhoods 
of $A_\delta(\vre)$ respectively,
and note that these sets are compact as well.

Now let $g: G \to [0,\infty)$ be a smooth function supported on $B(e,\vre/4)$
with $\|g\|_{L^1} = 1$, and take
$\psi = g \ast 1_{A_\delta(\vre)^+}$.
Then $\psi$ is supported on $A_\delta(\vre)^{++} \subset  U_\delta$, so properties (1) and (2) hold.
Furthermore, $\psi \equiv 1$ on $A_\delta(\vre)$, so
$$\int 1_{ U_\delta} \, d\mu \ge \int \psi\, d\mu \ge \mu\big(A_\delta(\vre)\big) 
\ge 2^{-(n+m+2)}c_1\delta^{n+m}.$$
Choosing $D_2 = c_1$ and $D_1 = 2^{-(n+m+2)}c_1$, we obtain property (3).

Let  $C \df \displaystyle \max_{h\in B(\0,\vre/4)} \ \|\nabla g(h)\|$.
Then, since $\mu(A_\delta(\vre)^+) \le 1$ for small enough $\delta > 0$, for every differential operator $D$ on $G$ 
we have \eq{diffineq}{|D\psi(\Lambda)| \le |(D g \ast 1_{A_\delta(\vre)^+})(\Lambda)|\,,} which implies (4).

Finally for $\ell\in\N$ let  $$C_\ell \df \displaystyle\frac{D_2}{D_1}\max_{0 \le |\alpha| \le \ell}\ \max_{h\in B(\0,\vre/4)} \ |D_\alpha g(h)|\,.$$
Using \equ{diffineq},  for small enough $\delta > 0$
we have $$\|D_\alpha \psi(\Lambda)\| \le |(D_\alpha g \ast 1_{A_\delta(\vre)^+})(\Lambda)|
		\le \frac{D_1}{D_2}C_\ell\,,$$
and Young's Inequality implies
$$\|D_\alpha \psi\|_{L^2} \le \|D_\alpha g\|_{L^2} \|1_{A_\delta(\vre)^+}\|_{L^1}.$$
Since, for small enough $\delta$, $D_\alpha \psi$ is supported on a set of measure less than $1$,
we get
$$\|\psi\|_\ell \le \frac{D_1}{D_2} C_\ell D_2 \delta^{m+n} \le C_\ell D_1 \delta^{n+m} \le C_\ell \|\psi\|_{L^1}.$$
\end{proof}

\begin{proof}[Proof of Proposition \ref{expdecaysets}] Recall that we are given $B = B({\bf 0},r/2)$, a small  $\delta > 0$,  $\Lambda$ satisfying \equ{inj} and $t\ge 0$.  
Take $\lambda$ and $\ell$ as in  Proposition \ref{expdecay},  let $\lambda' > 0$ be small enough that $$\lambda - (nm+\ell+1)\lambda' \ge \lambda'\,,$$
and let $\psi$ be as in Lemma \ref{Sobapprox3}.
Let $f = 1_B$ and define  $\vre\df {e^{-\lambda't}}$. We can assume that $\vre \le r$ since otherwise the right hand side of \equ{conclusion} is negative and the conclusion of the proposition follows. 

Take $f_\vre$  as in Lemma \ref{Sobapprox2}.
Then, by Proposition \ref{expdecay}, we have
$$\left|\int_Hf_\vre(A)\psi(g_thx)\, d\nu(h) - \int_H f_\vre\, d\nu \int_X \psi \, d\mu \right| 
			\le E\|f_\vre\|_{\ell}\| \psi\|_{\ell}\left(\max_{x\in X} \|\nabla \psi(x)\|   \int_H |f_\vre|\, d\nu \right)^ke^{-\lambda t},$$
so, letting $$E_1 = EM C_\ell C D_2C^k\big(\nu (B(\0,1))\big)^k$$ (which is
independent of $\delta$, $r$, and $t$), we have for $\delta < 1, r < 1/2$ and $t \ge 0$
$$\int_Hf_\vre(A)\psi(g_thx)\, d\nu(h) 
	\ge \int_H f_\vre\, d\nu \int_X \psi \, d\mu - E_1e^{(nm+\ell+1)\lambda't -\lambda t}
	\ge \int_H f_\vre\, d\nu \int_X \psi \, d\mu - E_1e^{-\lambda' t}.$$
Hence,
\begin{eqnarray*}
	\nu\big(\mathcal{A}(B,t,\delta,\Lambda)\big) &=& \int_H f(A)1_{ U_\delta}(g_thx)\, d\nu(h)\\
		&\ge& \int_H f(A)\psi(g_thx)\, d\nu(h) \\
		&\ge& \int_H f_\vre(A)\psi(g_thx)\, d\nu(h) - \int |f_\vre - f|\, d\nu\\
		&\ge& \int_H f_\vre(A)\psi(g_thx)\, d\nu(h) - 
				\nu\big(B({\mathbf 0}, r + e^{-\lambda' t}) \smallsetminus B({\mathbf 0}, r)\big)\\
		&\ge& \int_H f_\vre(A)\psi(g_thx)\, d\nu(h) - 
				nm2^{nm}e^{-\lambda' t}\\
		&\ge& \int_H f_\vre\, d\nu \int_X \psi\, d\mu 
				- E_1e^{-\lambda' t} - \m\n 2^{\n \m }e^{-\lambda' t}\\
		&\ge& D_1r^{\m\n }\delta^{\m +\n } - (E_1 + \m\n 2^{\m\n })e^{-\lambda' t}
\end{eqnarray*}
Taking $E' = E_1 + \m\n 2^{\m\n }$ and $D = D_1$ completes the proof.
\end{proof}
\ignore{Explanation: $\nu(B) = r^{mn}$, and $D\delta^{m+n}$ is pretty close to the measure of 
$U_\delta$, see \cite[\S 7]{KM2}. The exponent $a$ needs to be computed and will depend on $\ell$ and $k$ from the previous proposition. See also \cite{KM1}[Lemma 2.4.7].}

We will need to apply Proposition \ref{expdecaysets} to an arbitrary $\Lambda$ from the complement of $U_\delta$. This places a restriction on $r$, since we need to satisfy \equ{inj}. 
The following lemma gives us a concrete bound on how small $r$ must be in order to
meet this requirement.

\begin{lem}
\label{injrad} There exists $b > 0$ such that the injectivity radius of $X\ssm U_\delta$ is not less than $b \cdot \delta^{\m +\n }$. 
\end{lem}

\begin{proof}
By \cite[Theorem 4, p 72]{GK}\footnote{Theorem 4 in \cite{GK}
is stated in terms of the Eucliean norm but of course the result is also
valid for the sup norm, with a suitable adjustment of the constant.},
there is a constant $b'$ depending only on $n$ and $m$ such that
for any unimodular lattice $\Lambda \subset \R^{n+m}$
we may find a basis $(\vv_i)_{i=1}^{n+m}$ for $\Lambda$
with $\Pi_{i=1}^{n+m}\|\vv_i\| \le b' $.
If we assume that $\Lambda \notin U_\delta$, then this gives
$$\|\vv_k\| \le \frac{b'}{\Pi_{i \neq k} \|\vv_i\|} \le b' \delta^{-(n+m-1)}$$
for each $1 \le k \le m +n$.

\ignore{
Recall that the $k$th successive minimum $\lambda_k$ of $\Lambda$ is defined as the greatest lower bound of the numbers $R$ such that there exist $k$ linearly independent vectors of $\Lambda$ of length at most $R$. Equivalently, if $\Lambda = A \Z^d$ for some $A \in G$
and $K = A^{-1}B(\mathbf{0}, 1)$, we 
may take $\lambda_k$ to be the infimum of
the numbers $R$ such that there exist $k$ linearly independent vectors in $\Z^d \cap RK$.
According to a theorem of Minkowski, see e.g.\  
\cite[Theorem 1, p. 59]{GL}, 
$\prod_{k=1}^{\m+\n} \lambda_k \le 2^n/\vol\big(A^{-1} B(\mathbf{0}, 1)\big) = 1$, 
so since $\Lambda \notin U_\delta$ we have
$$
\lambda_{\m+\n} \le \frac1{\prod_{k=1}^{\m+\n-1} \lambda_k} \le \delta^{-(\m+\n-1)}.
$$
Thus, since ______, we may choose a basis $(\vv_i)_{i=1}^{\m+\n }$ of $\Lambda$
with $\|\vv_i \| \le  \delta^{-(\m+\n -1)}$
for each $1 \le i \le \m+\n $.
}

Now if $h_1,h_2\in G$ are such that $\|h_i - I_{\m+\n}\|_{op} < \frac{1}{4b'}\delta^{\m+\n }$ for $i = 1,2$
($\|\cdot\|_{op} $ here refers to the 
operator norm), 
then for each $1 \le k \le \m + \n$ one has $h_i \vv_k = \vv_k + \vu_i$, where 
$$\|\vu_i\| \le \|h_i - I_{\m+\n }\|_{op}  \cdot \|\vv_k\| \le  \frac1{4b'} \delta^{\m+\n } b'\delta^{-(\m+\n -1)}
									= \frac14\delta.$$
If $h_1\Lambda = h_2\Lambda$ then $h_1\vv_k - h_2\vv_k = \vu_1 - \vu_2 \in h_1\Lambda$.
But $\|\vu_1 - \vu_2\| \le \delta/2$, so, since $h_1\Lambda \notin U_{2\delta/3}$, we have a contradiction.
It remains only to observe that our distance function on $G$ satisfies
$\dist(h, I_{n+m}) \ge b'' \|h - I_{n+m}\|_{op} $, so choosing $b =\frac{b''}{4b'}$
completes the proof.
\end{proof}

We are now ready to prove Theorem \ref{upper}.

\begin{proof}[Proof of Theorem \ref{upper}]
Let $\vre = c^{\frac{1}{n+m}}$ and let $\delta = \vre/2$.
Note that by suitable change of the constant $k_2$, it suffices to prove
the statement for sufficiently small $c$, so assume
without loss of generality that
$c$ is small enough for Proposition \ref{expdecaysets} to hold and also small enough that
 \eq{def r}{r \df b\delta^{\m +\n } < 1/2\,,} where $b$ is as in Lemma \ref{injrad}.
Let
$B = B({\bold 0}, r/2)$ be the box of sidelength $r$ centered at $\bold 0\in\mr$;
we will estimate the dimension of $\big( \Bad_{\m ,\n }(c) - A_0\big)\cap B$ for an arbitrary $A_0\in\mr$.  Note that in view of Lemma \ref{danicorr}, 
$$ \Bad_{\m ,\n }(c) - A_0 = \{A\in\mr : g_t u_{A+A_0}\Z^{\m+\n} \notin U_\vre\ \forall\, t \ge 0\}.
$$
 
Now  by our choice of $\delta$ for any $\Lambda \notin U_\delta$ condition \equ{inj} will hold, 
so one can apply Proposition  \ref{expdecaysets} to get
\begin{equation*}
\nu\big(\mathcal{A}(B,t,\delta,\Lambda)\big) 
\ge  Dr^{\m \n }\delta^{\m +\n } -  E' e^{-\lambda
t}\,. 
\end{equation*}

Fix $t > 0$, let ${N} \df \left(\lfloor e^{\frac{(\m +\n )}{\m \n}t}\rfloor\right)^{\m\n}$, and break $B$ into 
$N$ cubes $B'$ of side length $r{N} ^{-1/\m \n }$. Note that for $t \ge \m\n$,
we have $N^{\frac{1}{\m\n}} \ge \frac12 e^{\frac{(\m +\n )}{\m\n}t}$, so
$$
r{N} ^{-1/\m \n } \le 2re^{-\frac{\m +\n }{\m \n }t}.$$ 
To estimate the number of subcubes of $B$ which intersect $\Bad_{\m ,\n }(c) - A_0$,
we need the following observation: if $A\in \mathcal{A}(B,t,\delta,\Lambda)$ (that is, $g_tu_A\Lambda \in U_\delta$), and 
one of the subcubes $B' \subset B$ contains $A$, then for any $A'\in B'$, 
\begin{equation}
\label{conjugate}
g_tu_{A'}\Lambda = g_tu_{A'-A}u_A\Lambda = (g_tu_{A'-A}g_{-t})g_tu_A\Lambda\,.
\end{equation}
Now, a straightforward calculation shows that
$${g_tu_{A'-A}g_{-t} = 
\begin{pmatrix}
   I_\m  & (A'-A)e^{\frac{\m+\n }{\m\n }t} \\
    0 & I_\n 
       \end{pmatrix}
}\,,
$$
so since $A$ and $A'$ lie in the same cube of side length $r{N} ^{-1/\n \m } \le 2re^{-\frac{\m +\n }{\m \n }t}$, 
we have
$\dist(g_t u_{A'-A}g_{-t}, I_{\m+\n }) \le 2r$.
Hence, for small enough $\delta > 0$ and any $\vv \in \Lambda$ we have
$$\|g_tu_{A'}\vv\| = \|(g_tu_{A'-A}g_{-t})g_tu_A\vv\| \le \|g_tu_A\vv\|(1+2r)
			\le \delta(1+2b\delta^{\m+\n }) \le 2\delta =\vre.$$
That is, $g_tu_{A'}\Lambda$ belongs to 
a slightly larger  set $U_\vre$, 
so we have proved that $B'\subset \mathcal{A}(B,t,\vre,\Lambda)$. Thus the measure of the union of subcubes $B'$ completely contained in $\mathcal{A}(B,t,\vre,\Lambda)$ is not less than the measure of $\mathcal{A}(B,t,\delta,\Lambda)$, so the number of those subcubes is at least 
$$
\frac{Dr^{\m \n }(\vre/2)^{\m +\n } - E' e^{-\lambda' t}}{r^{\m \n }/{N} } 
= {N}  \big(D'\vre^{\m +\n }- E' r^{-\m \n }e^{ -\lambda' t}\big)\,.
$$
Taking $\Lambda = u_{A_0}\Z^{\m +\n }$, 
which clearly does not belong to $U_\delta$, we see that the rest of the subcubes form a cover of the intersection of $B$ with 
$ \Bad_{\m ,\n }(c) - A_0$, 
and the number of elements of the cover is at most 
$$
{N} \big( 1 + E' r^{-\m \n }e^{-\lambda' t} - D'\vre^{\m +\n }\big)\,.$$
But now observe that by construction, for  any cube $B'$ from the above cover there exists  $A'\in B'$ such that $g_tu_{A'}\Lambda\notin U_\vre$. 
It follows that $g_tu_{A''}\Lambda \notin U_\delta$ for every $A'' \in B'$, since otherwise
by the computation above, we would have $g_tu_{A'}\Lambda\in U_\vre$.
In particular, denoting by $A_1$ the center of  $ B'$, we have
$g_tu_{A_1}\Lambda\notin U_\delta$.
Now note that conjugation by $g_t$ sends the cube of side length $r$ centered at the origin
to the cube of side length $e^{\frac{\m+\n }{\m\n }t}r$, so again using (\ref{conjugate})
we have
$$\{g_tu_{A'}\Lambda: A'\in B'\} = \{u_{A}g_tu_{A_1}\Lambda: A\in B\}\,.$$
Thus one can apply the same procedure to $g_tu_{A_1}\Lambda$ in place of $\Lambda$, 
getting a subdivision of $B$ into ${N} ^2$ cubes of side length $r{N} ^{-2/{\m \n }}$, and then cover $\big( \Bad_{\m ,\n }(c) - A_0\big)\cap B$ by ${N} ^2\left( 1 + E' r^{-\m \n }e^{-\lambda' t} - D'\vre^{\m +\n }\right)^2$ 
of those subcubes. Continuing inductively, we effectively embed $\big( \Bad_{\m ,\n }(c) - A_0\big)\cap B$ into a Cantor-like set, and therefore can conclude that
\begin{equation}
\label{generalupper}
\begin{split}\dim\big( \Bad_{\m ,\n }(c) - A_0\big)\cap B &\le \lim_{j\to\infty}\frac{j\log\left({N} (1 + E' r^{-\m \n }e^{-\lambda' t} - D'\vre^{\m +\n })\right)}{-\log(r{N} ^{-j/\m \n })}  \\
 = &\,\frac{\log\left({N} (1 + E' r^{-\m \n }e^{-\lambda' t} - D'\vre^{\m +\n })\right)}{-\log({N} ^{-1/\m \n })} \\
 \underset{\equ{def r}}{=} &\m \n  \left(1 + \frac{\log\left(1 + E''\vre^{-(\m +\n )\m \n }e^{-\lambda'
t} - D'\vre^{\m +\n }\right)}{(\m +\n )t}\right)\,.\end{split}\end{equation}

Note that this holds for all $t > \m\n$, so choose 
$t$ such that 
\eq{define t}{E''\vre^{-(\m +\n )\m \n }e^{-\lambda'
t} = \frac{D'\vre^{\m +\n }}2\,.
} (For small $\vre > 0$, this choice will satisfy $t > \m\n$.)
Then the right-hand side of (\ref{generalupper}) becomes
$$\m \n  \left(1 + \frac{\log\left(1 - \frac{D'}{2}\vre^{\m +\n }\right)}{(\m +\n )t}\right).$$
Now, for small enough $\vre$, $ \log\left( 1 - \frac{D'}2\vre^{\m +\n }\right) \le - C_1\cdot\vre^{\m +\n } = - C_1\cdot c $. Solving \equ{define t} for $t$, one gets $$t = \frac{1}{\lambda'}\left(\log\frac{2E''}{D'} + (\m +\n )(\m \n +1)\log \frac1\vre\right)\,,$$
which can be bounded by $C_2\log \frac1c$.
The desired estimate follows.
\end{proof}

\ignore{Now to write the best estimate it suffices to find $t$ which minimizes
$$
\varphi(t) \df  \frac{\log\left(1 + E''\vre^{-(m+n)(mn+a)}e^{-\lambda
t} - D'\vre^{m+n}\right)}{t}\,.
$$
I am too lazy to do this \textcolor{red}{(you are welcome to try!)} so let me simply produce some estimate, perhaps not optimal (although perhaps quite close to optimal). 

\begin{thm}
\label{upperestba} For any $m,n\in\N$, one has $$\codim\big(  Bad_{m, n}(c)\big) \ge \const(m,n)\cdot \frac{c^m}{\log(1/c)}\,.$$
\end{thm}
}

\bibliographystyle{alpha}

\end{document}